\documentclass[11pt]{article}
\usepackage[margin=1in]{geometry}
\usepackage{mathpazo} 
\usepackage{amsmath,amssymb}
\usepackage{amsthm}
\usepackage{booktabs}
\usepackage{tikz}
\usetikzlibrary{calc,positioning}
\usepackage{xcolor}

\title{\bfseries A Log-Log Saving for Matrix-Algebra Length and Terseness}
\author{ Florian Ito Sprung\footnote{Current address: IHES, Le Bois-Marie, 35 route de Chartres  CS 40001,  91893 Bures-sur-Yvette, France\\ E-mail: ian.sprung@gmail.com}}
\date{}

\usepackage{amssymb}
\usepackage{amscd,amsfonts,amsmath,amssymb, bbm,dsfont,extarrows}
\usepackage{amsmath}
\usepackage{amsfonts}
\usepackage{amssymb}\usepackage{enumerate,epsf,fancyhdr,float,graphicx,tabularx}
\usepackage{latexsym,mathrsfs,multirow}
\usepackage{wasysym}
\usepackage{xypic}
\usepackage[all]{xy}\usepackage[OT2,T1]{fontenc}
\usepackage[modulo,mathlines,displaymath, running]{lineno}
\usepackage{amsmath,mathrsfs,enumerate,enumitem,wasysym}
\usepackage{xypic,hhline}
\usepackage[all]{xy}
\usepackage[OT2,T1]{fontenc}
\usepackage[modulo,mathlines,displaymath]{lineno}
\usepackage{tikz,tikz-cd}
\usepackage{dashrule}
\usetikzlibrary{matrix}
\usepackage[modulo,mathlines,displaymath]{lineno}

\newtheorem{theorem}{Theorem}[section]
\DeclareSymbolFont{cyrletters}{OT2}{wncyr}{m}{n}\DeclareMathSymbol{\Sha}{\mathalpha}{cyrletters}{"58}

\renewcommand{\phi}{{\varphi}}

\renewcommand{\geq}{\geqslant}
\renewcommand{\leq}{\leqslant}

\newcommand{\smat}[1]{\left( \begin{smallmatrix} #1 \end{smallmatrix} \right)}

\newcommand{\links}{\left(\begin{array}{cc}}
\newcommand{\rechts}{\end{array}\right)}
\newcommand{\bai}{\left[\begin{array}{cc}}
\newcommand{\dai}{\end{array}\right]}
\newcommand{\hidari}{\left(\begin{array}{c}}
\newcommand{\migi}{\end{array}\right)}

\newcommand{\C}{\mathbb{C}}

\newcommand{\usim}{\sim_u}

\DeclareMathOperator{\tr}{tr}

\newcommand{\Mat}{\operatorname{Mat}}

\newcommand{\Span}{\operatorname{Span}}

\newcommand{\union}{\cup} 

\renewcommand{\contentsname}{Contents\\{\footnotesize\normalfont(A table
of contents should normally not be included)}}

\newtheorem{auxiliary proposition}[theorem]{Auxiliary Proposition}

\newtheorem{definition}[theorem]{Definition}

\newtheorem{lemma}[theorem]{Lemma}
\newtheorem{main conjecture}[theorem]{Main Conjecture}
\newtheorem{main theorem}[theorem]{Main Theorem}
\newtheorem{modesty proposition}[theorem]{Modesty Proposition}

\newtheorem{open problem}[theorem]{Open Problem}

\newtheorem{remark}[theorem]{Remark}

\newtheorem{convergence lemma}[theorem]{Convergence Lemma}
\newtheorem{corrected lemma}[theorem]{Corrected Lemma}
\newtheorem{growth lemma}[theorem]{Growth Lemma}
\newtheorem{coefficient lemma}[theorem]{Integrality Lemma}
\newtheorem{interpolation lemma}[theorem]{Interpolation Lemma}
\newtheorem{kernel lemma}[theorem]{Kernel Lemma}
\newtheorem{limit lemma}[theorem]{Limit Lemma}
\newtheorem{tandem lemma}[theorem]{Modesty Lemma}
\newtheorem{zero-finding lemma}[theorem]{Zero-Finding Lemma}

\newtheorem*{theoremA}{Theorem A}
\newtheorem*{theoremB}{Theorem B}

\numberwithin{equation}{section}

\begin{document}
\maketitle


\begin{abstract}
We study how long words in a family of matrices must be before they linearly span the algebra generated by that family. Let $F$ be a field, $S\subseteq \Mat_n(F)$, and $FS^{\leq k}$ the $F$-linear span of all words in $S$ of length at most $k$. Set
$\ell(S):=\min\left\{k:FS^{\leq k}=F[S]\right\}.$
\v{S}itov proved the estimate
$$ \ell(S)\leq 2n\log_2 n+4n-4. $$
We prove the log--log improvement
$$ \ell(S)\leq 2n\log_2 n-2n\log_2\log_2 n+5n $$
for every $n>1$ and every $S\subseteq \Mat_n(F)$.

A theorem of Specht gives a word-criterion for unitary similarity of complex $n \times n$ matrices.
The trace argument of Freedman–Gupta–Guralnick, as used by Pappacena, shows that our estimate can be used to bound the terseness $\tau(n)$, i.e. the shortest length of words needed in Specht’s theorem. Thus, for n > 1,
$$ \tau(n)\leq 4n\log_2 n-4n\log_2\log_2 n+10n+1,\qquad n>1. $$

\end{abstract}

\section{Introduction}

The length problem for matrix algebras asks how long one must multiply matrices until the resulting products span the \textit{algebra} generated by the given matrices as a \textit{vector space}.  More precisely, if $S\subseteq \Mat_n(F)$, let $F S^{\leq k}$ denote the $F$-linear span of all words in $S$ of length at most $k$, including the empty word (identity matrix). 
The length $\ell(S)$ is the least $k$ for which $FS^{\leq k}=F[S]$.  The  
length of the full matrix algebra is $L(\Mat_n(F)):=\max_{F[S]=\Mat_n(F)}\ell(S).$ We can now state our main result.

\begin{theoremA}
	For every field $F$, every $n>1$, and every $S\subseteq\Mat_n(F)$, 
	$$	\ell(S)\leq 2n\log_2 n-2n\log_2\log_2 n+5n.$$
In particular, $L(\Mat_n(F))$ satisfies the same bound.
\end{theoremA}


The length problem goes back to work of Paz, who proved the quadratic estimate $\ell(S)\leq \left\lceil\frac{n^2+2}{3}\right\rceil$
in \cite[Theorem~1 and Remark~2]{Paz1984}.  Paz also conjectured the linear bound $\ell(S)\leq 2n-2;$
 this is Conjecture~2 of \cite{Shitov2019}, see also \cite[Conjecture~1.6]{GutermanLaffeyMarkovaSmigoc2018}.  The conjecture remains open in general.

The same question also appears in representation theory and invariant theory in the work of Freedman, Gupta, and Guralnick on \v{S}ir\v{s}ov's theorem and representations of semigroups, who ask for a bounding function $g(n)$ in \cite[Question~3.6]{FreedmanGuptaGuralnick1997} for the analogous spanning degree for semigroups of matrices. Their Corollary~2.8 specializes this to the unitary-similarity problem for a matrix and its adjoint as in Specht's theorem. Specht's theorem \cite[Satz 1]{Specht1940} says that two complex $n \times n$ matrices $A$ and $B$ are unitarily similar if and only if $\tr w(A,A^*)=\tr w(B,B^*)$ for all words $w$.  

Inspired by this circle of ideas, Pappacena proved \cite{PappacenaThesis} a better bound for matrix length \cite[Corollary~3.2]{Pappacena1997}, giving the first subquadratic estimate
\begin{equation}
 \ell(S) < n\sqrt{\frac{2n^2}{n-1}+\frac14}+\frac n2-2 =\sqrt{2}\,n^{3/2}+O(n).\tag{*}
\end{equation}
Pappacena's crucial tool was a \textit{finishing lemma}, saying that if a short word span contains a matrix of a specified rank, that rank controls the number of additional letters in the words needed to span the full matrix algebra \cite[Theorem~4.1(a)]{Pappacena1997}. For another application of these length ideas in the context of affine semiprime algebras of Gelfand--Kirillov dimension one, see \cite[Theorem~5]{PappacenaSmallWald2003}. 

\v{S}itov \cite[Theorem~3 and Claim~14]{Shitov2019} sharpened the  estimate to
$$\ell(S)\leq 2n\log_2 n+4n-4 $$
by improving Pappacena's finishing lemma when the rank was one, and then crucially proving a \textit{descent} for square-zero matrices: Starting from a square-zero matrix, the descent repeatedly produces another square-zero matrix whose rank is at most half the previous rank, while keeping the `cost' for producing it low. The `cost' is the number of (extra) letters needed. \v{S}itov descends all the way to rank one, and then applies his improved finishing lemma. A careful reading of \cite[Proof of Theorem 3]{Shitov2019} then gives a slight improvement and results in the bound $$\ell(S)\leq 2n\log_2 n+2n-4. $$

The main observation of this paper is that it would be more judicious to stop \v{S}itov's square-zero rank descent as soon as the rank falls below $\sqrt{2}\log_2 n.$ If one stops the descent at a rank between $2$ and $\sqrt{2}\log_2 n$, we obtain the bound by simply applying Pappacena's finishing lemma\footnote{in fact, the bound is slightly better}. However, if the descent falls directly to rank one, then this means the previous rank was larger than $\sqrt{2}\log_2 n$. \v{S}itov's descent idea, in simple terms, was to discover a smaller rank matrix by looking at a `longer' space (i.e. generated by longer words), and then bounding the \textit{rank-length}, i.e. the product of the rank of the desired matrix and the length of the space it lives in at each step. To make this useful at the final step, we use a telescoping argument which relies on the elementary fact that for $x\in(0,\frac{1}{2}]$, we have $ 1-x\leq -\frac{1}{2}\log_2 x$. This is the origin of the $\log$. 
The final bound  is of the form 
$$5n+2n\log_2 n-2n\log_2 (\text{previous rank}).$$ 
Since the final rank is 1, the final length is the same as the final rank-length, and thus can be bounded by the above formula with the comparatively large previous rank, which we recall was bigger than  $\sqrt{2}\log_2 n$. It is this observation that is responsible for the log-log saving. 



The same bound has a direct consequence for the Specht--Pearcy trace invariants in the unitary-similarity problem.  Recall that Pearcy's theorem \cite[Theorem~1]{Pearcy1962} is an improvement on Specht's theorem,  saying that in Specht's criterion $\tr w(A,A^*)=\tr w(B,B^*),$ only words $w$ of length at most $2n^2$ have to be considered; see also Shapiro's discussion of Pearcy's bound \cite[p.~149]{Shapiro1991}. This gives way to the question how much the bound $2n^2$ can be improved: How \textit{terse} can the words be so that the conclusion of Specht's theorem still holds?
 Denoting by $A\usim B$ unitary similarity, we make the following definition:

\begin{definition}
The \textbf{terseness} $\tau(n)$ is the least integer $\geq0$ such that, for every $A,B\in\Mat_n(\C)$,  $A\usim B$ if and only if  $\tr w(A,A^*)=\tr w(B,B^*)$
for all words $w$ of length at most $\tau(n)$.
\end{definition}

In this terminology, Pearcy proved $\tau(n)\leq 2n^2$. If $g(n)$ satisfies
$\ell(S)\leq g(n)$ for every $S\subseteq\Mat_n(\C)$, then
Freedman--Gupta--Guralnick
\cite[Corollary~2.8]{FreedmanGuptaGuralnick1997} give us 
$\tau(n)\leq2g(n)+1$.  
 Thus, Theorem A gives the following consequence.

\begin{theoremB}
For $n> 1$,
$$\tau(n)\leq4n\log_2 n-4n\log_2\log_2 n+10n+1.$$
Consequently, the list of two-letter words needed via Specht's criterion has size at most $\left(\frac{n}{\log_2 n}\right)^{4n}2^{O(n)}.$

\end{theoremB}

\begin{remark}[Exact powers versus length]
There is a related, but different, problem where one studies the
linear span $FS^k$ of words of exactly\footnote{If the identity is
adjoined to the alphabet, then this problem is equivalent to ours, because we can pad every word of length at most $k$ by copies of
$I$:
$F(S\union\{I\})^k=FS^{\leq k}.$} length $k$. 
Without adjoining $I$, however, this exact-power problem is 
different, since the spaces $FS^k$ may not form an increasing filtration.
Micha{\l}ek and \v{S}itov proved an $O(n^2\log n)$ bound for the corresponding exact-power stabilization problem \cite[Theorem~1]{MichalekShitov2019}. 
A later preprint of
\v{S}itov proves that if
$FS^k=\Mat_n(F)$ holds for some $k$, then the least such $k$ is at most
$ n^2+2n-4 $
for $n>2$ \cite[Theorem~3]{ShitovGrowth2024}. 
\end{remark}

\section{Notation and known results}

Let $F$ be a field. We let all algebras be unital, and we let word spans include the empty word, which contributes the identity matrix.

\begin{definition}
Let $S\subseteq \Mat_n(F)$.  For $k\geq 0$, define the word span to be the $F$-vector space $$F S^{\leq k}:=\Span_F\{s_1s_2\cdots s_j:0\leq j\leq k,\ s_i\in S\}.$$
Let $F[S]$ denote the unital $F$-algebra generated by $S$.  Define the \textbf{length} $\ell(S)$ of $S$ to be the least $k$ such that $F S^{\leq k}=F[S].$

\end{definition}


Following\footnote{``Irreducible'' normally means that $S$ has no common
	nonzero proper invariant subspace.  The two definitions agree over algebraically closed fields.} \cite[1. Warm-Up]{Shitov2019}, we say that $S$ is \textit{irreducible} if $F[S]=\Mat_n(F)$. 
A result\footnote{See Burnside's 1905 theorem on irreducible linear groups \cite{Burnside1905}. For the statement in matrix-algebra form; see the statement on the first page of Radjavi--Rosenthal \cite[Theorem 1.5.1]{RadjaviRosenthal2000}.} of Burnside says that if $F$ is algebraically closed and $S$ is not irreducible, then the elements of $S$ can be simultaneously put into upper block-triangular form.
The following theorem bounds the length in terms of lengths coming from subblocks. This is \cite[Corollary~3]{Markova2005}, see also \cite[Lemma~4]{Shitov2019}. 

\begin{theorem}[Block reduction]
\label{thm:block-reduction}
Let $S\subseteq \Mat_n(F)$ be simultaneously block upper triangular, i.e. after a simultaneous change of basis, every $s\in S$ is of the form
$$ s=\smat{ s_p & * \\ 0 & s_q}$$
with $s_p\in \Mat_p(F)$, $s_q\in \Mat_q(F)$, and $p+q=n$. Put
$S_p:=\{s_p:s\in S\}, S_q:=\{s_q:s\in S\}.$

Then $\ell(S)\leq \ell(S_p)+\ell(S_q)+1.$
\end{theorem}




Pappacena's key estimate bounds $\ell(S)$ in terms of the rank of a matrix already found in a short word span. This is \cite[Theorem~4.1(a)]{Pappacena1997}.

\begin{theorem}[Pappacena's finishing lemma]
\label{thm:pappacena-rank}
Let $F$ be algebraically closed and $S\subseteq \Mat_n(F)$ be irreducible.  Suppose $F S^{\leq k}$ contains a matrix of rank $r>0$.  Then $\ell(S)\leq rn+n-r+k-1.
$
\end{theorem}


Pappacena also proved stronger conclusions when the generating set already contains a sufficiently convenient matrix, see \cite[Theorem~4.1(b),(c)]{Pappacena1997}. \v{S}itov addressed the excluded cases by showing the existence of square-zero matrices with lower and lower ranks that don't cost too much, i.e. with a controlled  number of letters. This is \textit{\v{S}itov's square-zero descent}:

\begin{theorem}[\v{S}itov's square-zero descent]
\label{thm:shitov-descent}
Let $F$ be algebraically closed and let $S\subseteq \Mat_n(F)$ be irreducible.
\begin{itemize} 
    \item There are nonzero integers $\lambda_0$ and $\rho_0$ so that $\lambda_0\rho_0\leq 2n$ and $F S^{\leq \lambda_0}$ contains a square-zero matrix of rank $\rho_0$.
     \item If $F S^{\leq \lambda_i}$ contains a square-zero matrix of rank $\rho_i\geq2$, then there are integers $\lambda_{i+1}$ and $\rho_{i+1}$ such that $F S^{\leq \lambda_{i+1}}$ contains a square-zero matrix of rank $\rho_{i+1}$, where $$1\leq \rho_{i+1}\leq \frac{\rho_i}{2} \text{ and }\lambda_{i+1}\leq\frac{\lambda_i\rho_i}{\rho_{i+1}}+\frac{4n(\rho_i-\rho_{i+1})}{\rho_i\rho_{i+1}}.$$
\end{itemize}
\end{theorem}

The initial square-zero matrix is \v{S}itov's Claim~11, and the descent step is \v{S}itov's Claim~14 \cite[Claims~11 and~14]{Shitov2019}.  We now spell out how the numerical estimate used later is extracted from this descent.

Apply Theorem~\ref{thm:shitov-descent} repeatedly, as long as the current rank is at least $2$.  Since the rank is at least halved at each step, this produces a finite sequence

$$(\lambda_0,\rho_0),\ (\lambda_1,\rho_1),\ldots,\ (\lambda_{t+1},\rho_{t+1}),$$
where $F S^{\leq \lambda_i}$ contains a square-zero matrix of rank $\rho_i$, and where $\rho_{t+1}=1$.  

Since $0<\frac{\rho_{i+1}}{\rho_{i}}\leq 1/2$, elementary calculus gives us
$ 1-\frac{\rho_{i+1}}{\rho_{i}}\leq -\frac{1}{2}\log_2 \frac{\rho_{i+1}}{\rho_{i}},$ so that 
\begin{equation}
	\label{eq:log-argument}
  \sum_{i=0}^j(1-\frac{\rho_{i+1}}{\rho_{i}})\leq\frac12\log_2\frac{\rho_0}{\rho_{j+1}}.
\end{equation}

 It is useful to multiply the word length by the current rank and set $\mu_i:=\lambda_i\rho_i.$ Thus $\mu_i$ is the \emph{rank-length} at the $i$-th stage of the descent.
 We  have  $ \mu_{i+1} \leq \mu_{i}+4n\left(1-\frac{\rho_{i+1}}{\rho_{i}}\right).$
Iterating this inequality and using $\mu_0\leq 2n$ gives
$ \mu_{j+1} \leq 2n+4n\sum_{i=0}^j\left(1-\frac{\rho_{i+1}}{\rho_{i}}\right).$ 
Applying equation (\ref{eq:log-argument}) to this, we obtain $
  \mu_{j+1}\leq2n+2n\log_2\frac{\rho_0}{\rho_{j+1}}.$
Since every square-zero $n\times n$ matrix has rank at most $n/2$, we have $\rho_0\leq n/2$. 
Thus, we have proved the following lemma:
\begin{lemma}\label{lem:mu-bound}The following inequality holds\footnote{It even holds at the zero index: $\mu_0\leq2n\leq2n\log_2\frac{n}{\rho_0}$}: $$\mu_{j+1} \leq 2n+2n\log_2\frac{n}{2\rho_{j+1}}=2n\log_2\frac{n}{\rho_{j+1}}.$$
\end{lemma}

Putting $j=t$ and remembering that $\rho_{t+1}=1$, Lemma \ref{lem:mu-bound} then gives us $ \lambda_{t+1}\leq 2n\log_2n.$ Combining this with a finishing lemma\footnote{\v{S}itov improves Pappacena's finishing lemma (Theorem \ref{thm:pappacena-rank}) in the rank one case from Pappacena's $\ell(S)\leq 2n+k-2$ to
$\ell(S)\leq 2n+k-4,$ cf. \cite[Corollary~7]{Shitov2019}.} culminates in \cite[Theorem~3]{Shitov2019}:

\begin{theorem}[\v{S}itov, slightly improved\footnote{\v{S}itov's published bound is the slightly worse $\ell(S)\leq 2n\log_2 n+4n-4$. 
Note that \v{S}itov ends \cite{Shitov2019} by writing `this paper does not show any effort on improving the $o(n \log n)$ part of the upper bound.' }]
\label{thm:shitov-global}
For every field $F$, every $n\geq2$, and every $S\subseteq\Mat_n(F)$, $$\ell(S)\leq2n\log_2n+2n-4. $$
\end{theorem}


We also record the following consequence.


\begin{theorem}\cite[Corollary~2.8]{FreedmanGuptaGuralnick1997}
\label{thm:specht-transfer}
Suppose $g(n)$ is a function such
that $\ell(S)\leq g(n)$ for every $S\subseteq\Mat_n(\C)$.  Then
$$ \tau(n)\leq2g(n)+1. $$ In words, any general upper bound for the length of all subsets of $\Mat_n(\C)$ gives a terseness bound of twice that bound plus one.
\end{theorem}


\section{The new result and proof}

We now prove the advertised improvement.  Give our bound a name by putting $$B(x):=2x\log_2 x-2x\log_2\log_2 x+5x\qquad(x\geq 2),$$ and set $B(1):=0$.  We need the following elementary fact to reduce the argument to the irreducible case, cf. the block reduction Theorem \ref{thm:block-reduction}.

\begin{lemma}
\label{lem:subadditive-B}
If $p,q\geq 1$ are integers and $p+q\geq 16$, then $ B(p)+B(q)+1\leq B(p+q).
$

\end{lemma}

Assuming this lemma, we now prove Theorem A.

\begin{theoremA}
Let $F$ be a field, $n>1$ be a positive integer, and $S\subseteq\Mat_n(F)$.  Then $$\ell(S)\leq2n\log_2 n-2n\log_2\log_2 n+5n.$$
In particular, $L(\Mat_n(F))\leq 2n\log_2 n-2n\log_2\log_2 n+5n$.
\end{theoremA}

\begin{proof}

We induct on $n$, uniformly
over all fields and all matrix families, i.e. the $n$th induction step assumes that for every field $K$, every $m<n$, and every
$T\subseteq\Mat_m(K)$, one has $\ell(T)\leq B(m)$. 

The cases $2\leq m<16$ follow from Pappacena's general bound $(*)$ from the introduction, see the table in the last section. So we assume $n\geq 16$. Let $S\subseteq \Mat_n(F)$. Let $\overline F$ be an algebraic
closure of $F$. We have $ \overline F\otimes_F F[S]\cong\overline F[S], \overline F\otimes_F FS^{\leq k}\cong\overline F S^{\leq k}.$
Faithful flatness then gives $FS^{\leq k}=F[S]$ if and only if $\overline F S^{\leq k}=\overline F[S]$, so
$\ell(S)$ is unchanged.  We consequently assume from now on that $F=\overline{F}$.

If $S$ is reducible, then by simultaneous triangularization there is a nontrivial block upper triangular form.  Let the diagonal block sizes be $p$ and $q$, with $p+q=n\geq 16$.  By Theorem~\ref{thm:block-reduction} and the induction hypothesis, $\ell(S) \leq B(p)+B(q)+1  \leq B(n)$ by Lemma~\ref{lem:subadditive-B}. 

 Thus we may assume that $F$ is algebraically closed and that $S$ is irreducible.  By Burnside's theorem, this means that $F[S]=\Mat_n(F)$, so the rank estimates and square-zero descent apply as in \cite{Shitov2019}. 

Put $$L:=\log_2 n, \text{ and } d(L):=\frac{B(n)}{n}=2L-2\log_2L+5.$$
Since $n\geq 16$, we have $L\geq 4$.  By the first step  in \v{S}itov's descent (Theorem~\ref{thm:shitov-descent}) there is a nonzero square-zero matrix of rank $\rho_0$ in $F S^{\leq \lambda_0}$ with $\lambda_0\rho_0\leq 2n.$

Pappacena's finishing lemma (Theorem \ref{thm:pappacena-rank}) gives
$$\ell(S)\leq \rho_0n+n+\lambda_0-\rho_0-1\leq n\left(\rho_0+1+\frac2\rho_0\right).
$$ When $\rho_0\leq \sqrt2 L$, this proves the theorem. Indeed, the function $x\mapsto x+1+2/x$ is convex, so on any closed interval,  it is
bounded by its endpoint values.  At $x=1$ it equals $4\leq d(L)$,
while at $x=\sqrt2 L$ it is at most
$\sqrt2 L+1+\sqrt2\leq d(L)$. Hence $\ell(S)\leq nd(L)=B(n)$ on the interval $[1,\sqrt2 L]$.


We therefore assume $\rho_0>\sqrt2 L$. Starting from $(\lambda_0,\rho_0)$, apply the square-zero descent repeatedly, but only as long as the current rank is greater than $\sqrt2 L$.  
Let $j$ be the first index such that $\rho_{j+1}\leq \sqrt2 L.$
There are two cases: either $2\leq \rho_{j+1}\leq \sqrt2 L$, or $\rho_{j+1}=1$.

\textbf{Case 1.} First suppose that $2\leq \rho_{j+1}\leq \sqrt2 L.$
By Lemma \ref{lem:mu-bound},
$$\lambda_{j+1}=\frac{\mu_{j+1}}{\rho_{j+1}}\leq\frac{2n\log_2(n/\rho_{j+1})}{\rho_{j+1}} =\frac{2n(L-\log_2\rho_{j+1})}{\rho_{j+1}}.$$
Using Pappacena's finishing lemma, Theorem~\ref{thm:pappacena-rank},
$$\ell(S)\leq n\rho_{j+1}+n+\lambda_{j+1}-\rho_{j+1}-1 \\\leq n\rho_{j+1}+n+\frac{2n(L-\log_2\rho_{j+1})}{\rho_{j+1}}.$$

Set $f_L(x):=x+1+\frac{2(L-\log_2 x)}{x}$. A direct differentiation gives $f_L''(x) =\frac{4(L-\log_2x)+6/\ln2}{x^3}>0$ for $2\leq x\leq \sqrt2 L$,
so $f_L$ is convex and it is enough to check the endpoints in $[2,\sqrt2 L]$.  At
$x=2,f_L(2)=L+2\leq d(L),$
because $L-2\log_2L+3\geq3$.  At $x=\sqrt2 L,$ $$f_L(\sqrt2 L)\leq \sqrt2 L+1+\frac{2L}{\sqrt2 L}=\sqrt2 L+1+\sqrt2\leq d(L).$$
Thus, $\ell(S)\leq nd(L)=B(n)$.



This is the desired estimate in the first case.

\textbf{Case 2.} It remains to treat the second case, the overshoot case $\rho_{j+1}=1.$

  By definition (minimality) of ${j}$, we have $\rho_j>\sqrt{2}L$.  The final descent step increases the rank-length $\mu$ by at most $4n$: indeed,
$\mu_{j+1}\leq\mu_{j}+4n\left(1-\frac{1}{\rho_j}\right)\leq\mu_{j}+4n.$
Using Lemma~\ref{lem:mu-bound} at the previous rank $\rho_j$,

$\mu_{j}\leq2n\log_2\frac{n}{\rho_j}\leq2n(L-\log_2 \sqrt2 L)=2nL-2n\log_2L-n,$
since $\rho_j>\sqrt2 L$. 
Since $\rho_{j+1}=1$, we have $\lambda_{j+1}=\mu_{j+1}$.  Hence
$\lambda_{j+1}\leq \mu_j+4n\leq 2nL-2n\log_2L-n+4n.$
The $r=1$ case of Pappacena's finishing lemma (Theorem \ref{thm:pappacena-rank}) now gives $$\ell(S)\leq2n+\lambda_{j+1}-2\leq2nL-2n\log_2 L+5n=B(n).$$

This proves the theorem.

\end{proof}

\begin{remark}
	The third term in our formula, $5n$, is optimal among integer multiples of $n$, if one stops the \v{S}itov descent at a constant multiple of $L$. We chose the multiple $\sqrt{2}L$ for convenience\footnote{Some exercises to the reader: Stopping at $aL$ gives the
coefficient $ 6-2\log_2 a$. The calculations become simpler when working simply with the trivial multiple $L$, where the third term becomes $6n$. 
}.
\end{remark}

The consequence concerning the theorems of Specht and Pearcy is now immediate.

\begin{theoremB}
Let $n>1$.  Then
$$\tau(n)\leq4n\log_2 n-4n\log_2\log_2 n+10n+1.$$

\end{theoremB}

\begin{proof}
Apply Theorem~\ref{thm:specht-transfer} to Theorem A over $F=\C$.

\end{proof}

It remains to prove Lemma \ref{lem:subadditive-B}.

\begin{proof}[Proof of Lemma \ref{lem:subadditive-B}]
Write
$B(x)=x c(x)$ with
$$c(x):=2\log_2x-2\log_2\log_2x+5.$$
For $x\geq3,c'(x)=\frac{2}{x\ln2}\left(1-\frac1{\ln x}\right)>0,$ so $c$ is increasing, and therefore $c(n)\geq c(16)=9\text{ for }n\geq16.$
Moreover, $B''(x)=\frac{2}{x\ln2}\left(1-\frac1{\ln x}+\frac1{(\ln x)^2}\right)>0\text{for}x>1,$ so $B$ is strictly convex.

Put $n=p+q$.  If one of $p,q$ is $1$, say $p=1$, then $B(n)-B(n-1)=c(n)+(n-1)\bigl(c(n)-c(n-1)\bigr)\geq c(n)\geq9>1,$ which proves the claim.

Now suppose $p,q\geq2$.  The function $x\mapsto B(x)+B(n-x)$ is convex on $[2,n-2]$ (this comes down to differentiating $B(x)+B(n-x)$ twice), and hence its maximum is attained at an
endpoint.  Thus $B(p)+B(q)\leq B(2)+B(n-2).$
Since $B(2)=14$ and $B(n)-B(n-2)=2c(n)+(n-2)\bigl(c(n)-c(n-2)\bigr)\geq2c(n)\geq18>15,$ we obtain $B(p)+B(q)+1\leq B(2)+B(n-2)+1\leq B(n).$

\end{proof}

\section{A table comparing the bounds}

The following two tables compare the corresponding integer degree
cutoffs for {\small
$$ B(n):=2n\log_2 n-2n\log_2\log_2 n+5n,
B_{\mathrm{Pap}}(n):=n\sqrt{\frac{2n^2}{n-1}+\tfrac14}+\tfrac n2-2,
\text{ and }B_{\mathrm{Sh}}(n):=2n\log_2 n+2n-4. $$}

\begin{center}
\scriptsize
\setlength{\tabcolsep}{3pt}
\begin{tabular}{rrrr|rrrr|rrrr}
\toprule
$n$ & $\lfloor B\rfloor$ & $\lfloor B_{\mathrm{Pap}}\rfloor$ & $\lfloor B_{\mathrm{Sh}}\rfloor$ &
$n$ & $\lfloor B\rfloor$ & $\lfloor B_{\mathrm{Pap}}\rfloor$ & $\lfloor B_{\mathrm{Sh}}\rfloor$ &
$n$ & $\lfloor B\rfloor$ & $\lfloor B_{\mathrm{Pap}}\rfloor$ & $\lfloor B_{\mathrm{Sh}}\rfloor$ \\
\midrule
$2$ & $14$ & $\mathbf{4}$ & $\mathbf{4}$
&
$7$ & $53$ & $\mathbf{30}$ & $49$
&
$12$ & $101$ & $\mathbf{65}$ & $106$ \\

$3$ & $20$ & $\mathbf{8}$ & $11$
&
$8$ & $62$ & $\mathbf{36}$ & $60$
&
$13$ & $112$ & $\mathbf{73}$ & $118$ \\

$4$ & $28$ & $\mathbf{13}$ & $20$
&
$9$ & $72$ & $\mathbf{43}$ & $71$
&
$14$ & $122$ & $\mathbf{82}$ & $130$ \\

$5$ & $36$ & $\mathbf{18}$ & $29$
&
$10$ & $81$ & $\mathbf{50}$ & $82$
&
$15$ & $133$ & $\mathbf{90}$ & $143$ \\

$6$ & $44$ & $\mathbf{23}$ & $39$
&
$11$ & $91$ & $\mathbf{57}$ & $94$
&
$16$ & $144$ & $\mathbf{99}$ & $156$ \\
\bottomrule
\end{tabular}
\end{center}

The next table records some powers of two.  We also include $n=63$, since the integer cutoffs supplied by $B(n)$ and $B_{\mathrm{Pap}}(n)$ agree there, while $B(n)$ becomes strictly smaller at $n=64$.

\begin{center}
\scriptsize
\setlength{\tabcolsep}{3pt}
\begin{tabular}{rrrr|rrrr|rrrr}
\toprule
$n$ & $\lfloor B\rfloor$ & $\lfloor B_{\mathrm{Pap}}\rfloor$ & $\lfloor B_{\mathrm{Sh}}\rfloor$ &
$n$ & $\lfloor B\rfloor$ & $\lfloor B_{\mathrm{Pap}}\rfloor$  & $\lfloor B_{\mathrm{Sh}}\rfloor$ &
$n$ & $\lfloor B\rfloor$  & $\lfloor B_{\mathrm{Pap}}\rfloor$  & $\lfloor B_{\mathrm{Sh}}\rfloor$ \\
\midrule
$32$ & $331$ & $\mathbf{274}$ & $380$
&
$128$ & $\mathbf{1{,}713}$ & $2{,}119$ & $2{,}044$
&
$1{,}024$ & $\mathbf{18{,}796}$ & $46{,}876$ & $22{,}524$ \\

$63$ & $\mathbf{743}$ & $\mathbf{743}$ & $875$
&
$256$ & $\mathbf{3{,}840}$ & $5{,}931$ & $4{,}604$
&
$2{,}048$ & $\mathbf{41{,}126}$ & $132{,}130$ & $49{,}148$ \\

$64$ & $\mathbf{757}$ & $760$ & $892$
&
$512$ & $\mathbf{8{,}529}$ & $16{,}656$ & $10{,}236$
&
$4{,}096$ & $\mathbf{89{,}415}$ & $372{,}824$ & $106{,}492$ \\
\bottomrule
\end{tabular}
\end{center}


\section*{Acknowledgments}
 The author thanks the Institut des Hautes Etudes Scientifiques for its hospitality and excellent working conditions. 
 

The author was supported by Simons grant MPS-TSM-00008075.

\bibliographystyle{plain}
\bibliography{mainloglog}

@misc{ShitovGrowth2024,
  author = {Shitov, Yaroslav},
  title = {Growth in Matrix Algebras and a Conjecture of
           {P{\'e}rez-Garc{\'i}a, Verstraete, Wolf and Cirac}},
  year = {2024},
  note = {Preprint},
  doi = {10.13140/RG.2.2.22024.33280}
}

@article{Specht1940,
  author  = {Specht, Wilhelm},
  title   = {Zur {T}heorie der {M}atrizen. {II}},
  journal = {Jahresbericht der Deutschen Mathematiker-Vereinigung},
  volume  = {50},
  pages   = {19--23},
  year    = {1940}
}

@article{Pearcy1962,
  author  = {Pearcy, Carl},
  title   = {A complete set of unitary invariants for operators generating finite {W}*-algebras of type {I}},
  journal = {Pacific Journal of Mathematics},
  volume  = {12},
  pages   = {1405--1416},
  year    = {1962}
}

@article{Shapiro1991,
  author  = {Shapiro, Helene},
  title   = {A survey of canonical forms and invariants for unitary similarity},
  journal = {Linear Algebra and its Applications},
  volume  = {147},
  pages   = {101--167},
  year    = {1991}
}

@article{Paz1984,
  author  = {Paz, Azriel},
  title   = {An application of the {Cayley--Hamilton} theorem to matrix polynomials in several variables},
  journal = {Linear and Multilinear Algebra},
  volume  = {15},
  number  = {2},
  pages   = {161--170},
  year    = {1984}
}

@article{FreedmanGuptaGuralnick1997,
  author  = {Freedman, Allen R. and Gupta, R. N. and Guralnick, Robert M.},
  title   = {Shirshov's theorem and representations of semigroups},
  journal = {Pacific Journal of Mathematics},
  volume  = {181},
  number  = {3},
  pages   = {159--176},
  year    = {1997},
  doi     = {10.2140/pjm.1997.181.159}
}

@article{Pappacena1997,
  author  = {Pappacena, Christopher J.},
  title   = {An upper bound for the length of a finite-dimensional algebra},
  journal = {Journal of Algebra},
  volume  = {197},
  number  = {2},
  pages   = {535--545},
  year    = {1997},
  doi     = {10.1006/jabr.1997.7140}
}

@article{GutermanLaffeyMarkovaSmigoc2018,
  author  = {Guterman, Alexander and Laffey, Thomas and Markova, O. and {\v{S}}migoc, Helena},
  title   = {A resolution of {Paz}'s conjecture in the presence of a nonderogatory matrix},
  journal = {Linear Algebra and its Applications},
  volume  = {543},
  pages   = {234--250},
  year    = {2018}
}

@article{Shitov2019,
  author  = {Shitov, Yaroslav},
  title   = {An improved bound for the lengths of matrix algebras},
  journal = {Algebra \& Number Theory},
  volume  = {13},
  number  = {6},
  pages   = {1501--1507},
  year    = {2019},
  doi     = {10.2140/ant.2019.13.1501}
}

@article{Burnside1905,
  author  = {Burnside, William},
  title   = {On the condition of reducibility of any group of linear substitutions},
  journal = {Proceedings of the London Mathematical Society},
  series  = {2},
  volume  = {3},
  number  = {1},
  pages   = {430--434},
  year    = {1905},
  doi     = {10.1112/plms/s2-3.1.430}
}

@article{Markova2005,
  author  = {Markova, O. V.},
  title   = {On the length of upper-triangular matrix algebra},
  journal = {Russian Mathematical Surveys},
  volume  = {60},
  number  = {5},
  pages   = {984--985},
  year    = {2005}
}

@article{MichalekShitov2019,
  author  = {Micha{\l}ek, Mateusz and Shitov, Yaroslav},
  title   = {Quantum version of {Wielandt}'s inequality revisited},
  journal = {IEEE Transactions on Information Theory},
  volume  = {65},
  number  = {8},
  pages   = {5239--5242},
  year    = {2019},
  doi     = {10.1109/TIT.2019.2897772}
}

@phdthesis{PappacenaThesis,
  author = {Pappacena, Christopher J.},
  title  = {Matrix pencils and a generalized {Clifford} algebra},
  school = {University of Southern California},
  year   = {1998},
  note   = {Ph.D. thesis}
}

@book{RadjaviRosenthal2000,
  author    = {Radjavi, Heydar and Rosenthal, Peter},
  title     = {Simultaneous Triangularization},
  series    = {Universitext},
  publisher = {Springer-Verlag},
  address   = {New York},
  year      = {2000},
  doi       = {10.1007/978-1-4612-1200-3},
  isbn      = {978-0-387-98466-7}
}

@article{PappacenaSmallWald2003,
  author  = {Pappacena, Christopher J. and Small, Lance W. and Wald, Jeanne},
  title   = {Affine semiprime algebras of {GK} dimension one are (still) {PI}},
  journal = {Glasgow Mathematical Journal},
  volume  = {45},
  number  = {2},
  pages   = {243--247},
  year    = {2003},
  doi     = {10.1017/S0017089503001204},
  eprint  = {math/0211330},
  archivePrefix = {arXiv}
}

\end{document}